\documentclass[a4paper,12pt,reqno]{amsart}

\usepackage{amsmath,amsthm}
\usepackage{amssymb,latexsym}
\usepackage{amsaddr}

\newtheorem{thm}{Theorem}[section]

\title{A note on the minimum size of $k$-rainbow connected graphs.}

\author{Allan Lo}
\address{School of Mathematics, University of Birmingham, Birmingham, B15~2TT, United Kingdom}
\email{s.a.lo@bham.ac.uk}
\thanks {The research leading to these results was supported by the European Research Council
under the ERC Grant Agreement no. 258345.}
\date{\today}
\begin{document}

\begin{abstract}
An edge-coloured graph $G$ is \emph{rainbow connected} if there exists a rainbow path between any two vertices.
A graph $G$ is said to be \emph{$k$-rainbow connected} if there exists an edge-colouring of $G$ with at most $k$ colours that is rainbow connected.
For integers $n$ and~$k$, let $t(n,k)$ denote the minimum number of edges in $k$-rainbow connected graphs of order~$n$.
In this note, we prove that $t(n,k) = \lceil k(n-2)/(k-1) \rceil$ for all $n, k \ge 3$.
\end{abstract}

\maketitle

\section{Introduction}
We consider finite and simple graphs only.
An edge-coloured graph is \emph{rainbow} if all edges have distinct colours.
An edge-coloured graph is \emph{rainbow connected} if there exists a rainbow path between any two vertices.
Given an integer $k$, a graph $G$ is \emph{$k$-rainbow connected} if there is an edge-colouring of $G$ with at most $k$ colours that is rainbow connected.
This notion of connectivity was first introduced by Chartrand, Johns, McKeon and Zhang~\cite{MR2400153} in~2008.
Since then, many results have been discovered.
For a survey, we recommend~\cite{MR3015944}.

For integers $n$ and~$k$, let $t(n,k)$ denote the minimum number of edges in $k$-rainbow connected graphs of order~$n$.
Schiermeyer~\cite{MR3015314} evaluated $t(n,k)$ exactly for $k=1$ and $k \ge n/2$.
\begin{thm}[Schiermeyer~\cite{MR3015314}] \label{thm:bound}
\begin{align*}
t(n,k) = 
\begin{cases}
\binom{n}2 & \text{for $k = 1$,}\\
n & \text{for $n/2 \leq k \leq n-2$,}\\
n-1 & \text{for $ k \geq n-1$.}
\end{cases}
\end{align*}
\end{thm}
In the same paper, he also showed that $t(n,2) = (1 + o(1))n \log_2 n$. 
The lower bound was further improved by Li, Li, Sun and Zhao~\cite{li2012note}.
For general $3 \le k < n/2$, the best known bounds on $t(n,k)$ are
\begin{align}
	\left\lceil \frac{(k+1)n-1}k \right\rceil -k-2 \le 	t(n,k) \le \left\lceil \frac{k(n-2)}{k-1} \right\rceil, \label{eqn:bound}
\end{align}
where the lower bound is due to Li et al.~\cite{li2012note} and the upper bound is due to a construction of Bode and Harborth~\cite{BodeHarborth}.
When $k = 3$, Bode and Harborth~\cite{BodeHarborth} showed that $t(n,3)$ is actually equal to the upper bound for $n \ge 3$.
In this note, we show that the same statement holds for all $n,k \ge 3$.

\begin{thm} \label{thm:1-connected}
For $k , n \ge 3$, we have $t(n,k) = \lceil k(n-2)/(k-1) \rceil $.
\end{thm}

For $n/2 < k$, this theorem coincide with Theorem~\ref{thm:bound}.
As mentioned before, the case $k =3$ has been already proved by Bode and Harborth~\cite{BodeHarborth}, but our proof is different and shorter.

We would need the following notation.
For (edge-coloured) graphs $G$ and disjoint $U,W \subseteq V(G)$, we write $G[U]$ for the (edge-coloured) subgraph of $G$ induced by~$U$ and $G[U,W]$ for the (edge-coloured) bipartite subgraph of $G$ induced by partition classes $U$ and~$W$.

\begin{proof}[Proof of Theorem~\ref{thm:1-connected}]
Note that $t(n,k) \le \lceil k(n-2)/(k-1) \rceil $ by Theorem~\ref{thm:bound} and~\eqref{eqn:bound}.
Therefore, to prove the theorem, it suffices to show that $t(n,k) \ge k(n-2)/(k-1)$ for all $n, k \ge 3$.
Fix $k \ge 3$.
Suppose the theorem is false, so there exists a $k$-rainbow connected graph $G$ of order $n$ with $e(G) < k(n-2)/(k-1)$, so $n > 2k$ by Theorem~\ref{thm:bound}.
We further assume that $n$ is minimal.
Fix an edge-colouring~$c$ of $G$ with colours $\{1,2,\dots,k\}$ such that the resultant edge-coloured graph~$G^c$ is rainbow connected.
Without loss of generality, there are at least $e(G)/k$ edges of colour~$k$.
We are going to show that there exists a tripartition $V_1,V_2,V_3$ of $V(G)$ such that, for all $1 \le i< j \le 3$,
\begin{itemize}
	\item[(i)] all edges between $V_i$ and $V_j$ have colours $k$ in $G^c$;
	\item[(ii)] $G[V_i \cup V_j]$ is rainbow $k$-connected; 
	\item[(iii)] there is an edge between $V_i$ and $V_j$ in $G$.
\end{itemize}
Let $H$ be the edge-coloured subgraph obtained from $G^c$ by removing all the edges of colour $k$.
Note that $e(H) \le e(G) - e(G)/k < n-2$.
Hence, $H$ has at least $3$ components.
Let $V_1,V_2,V_3$ be a tripartition of $V(G)$ such that $H[V_i,V_j]$ is empty for all $1 \le i < j \le 3$ and $V_i \ne \emptyset$ for all~$1 \le i \le 3$.
(Note that $H[V_i]$ may consist of more than one components.)
Fix $1 \le i< j \le 3$.
Clearly, (i) holds by our construction.
To show that (ii) holds, it suffices to show that $G^c[V_i \cup V_j]$ is rainbow connected.
Recall that $G^c$ is rainbow connected, so for all $x, y \in V_i \cup V_j$, there exists a rainbow path $P$  in $G^c$ from $x$ to $y$.
By~(i), we deduce that $V(P) \subseteq V_i \cup V_j$.
Therefore (ii) holds.
Moreover, (iii) holds by considering a rainbow path $P$ in $G^c$ from $x \in V_i$ to $y \in V_j$.
Thus, we have the desired tripartition of~$V(G)$.

For $1 \le i \le 3$, let $n_i = |V_i|$ and so we have $n_i \ge 1$ by~(iii) and $n_1 + n_2 + n_3 = n$.
Since $n$ is chosen to be minimal, (ii) implies that $e(G[V_i \cup V_j]) \ge k(n_i+n_j -2)/(k-1)$ for all $1 \le i < j \le 3$.
Recall (iii) that $e(G[V_i,V_j]) \ge 1$.
Therefore we have
\begin{align*}
	2e(G)  & = \sum_{1 \le i < j \le 3} \Big( e(G[V_i \cup V_j]) + e(G[V_i,V_j]) \Big) \\
	& \ge \sum_{1 \le i < j \le 3} \left( \frac{k(n_i+n_j -2)}{k-1} + 1 \right) 
	= \frac{2k(n - 3)}{k-1} +3\ge  \frac{2k(n - 2)}{k-1},
\end{align*}
where the last inequality holds since $k \ge 3$.
Thus, $e(G) \ge k(n-2)/(k-1)$, a contradiction.
\end{proof}

\end{document}